\documentclass{amsart}

\usepackage{amsmath, amssymb,epic,graphicx,mathrsfs,enumerate}

\newtheorem{lemma}{Lemma}
\newtheorem{teor}{Theorem}

\newtheorem{prop}{Proposition}

\newtheorem{cor}{Corollary}

\DeclareMathOperator{\PSL}{PSL}
\DeclareMathOperator{\PGL}{PGL}
\DeclareMathOperator{\SL}{SL}
\DeclareMathOperator{\GL}{GL}

\begin{document}

\title{Group partitions of minimal size}
\date{}

\author{Martino Garonzi}
\address[Martino Garonzi]{Departamento de Matem\'atica, Universidade de Bras\'ilia, Campus Universit\'ario Darcy Ribeiro, Bras\'ilia-DF, 70910-900, Brazil}
\email{mgaronzi@gmail.com}
\thanks{The authors acknowledge the support of the Funda\c{c}\~{a}o de Apoio \`a Pesquisa do Distrito Federal (FAPDF) and the Coordena\c{c}\~{a}o de Aperfei\c{c}oamento de Pessoal de N\'ivel Superior - Brasil (CAPES) - Finance Code 001.}

\author{Michell Lucena Dias}
\address[Michell Lucena Dias]{Departamento de Matem\'atica, Universidade de Bras\'ilia, Campus Universit\'ario Darcy Ribeiro, Bras\'ilia-DF, 70910-900, Brazil}
\email{M.L.Dias@mat.unb.br}

\begin{abstract}
A cover of a finite group $G$ is a family of proper subgroups of $G$ whose union is $G$, and a cover is called minimal if it is a cover of minimal cardinality. A partition of $G$ is a cover such that the intersection of any two of its members is $\{1\}$. In this paper we determine all finite groups that admit a minimal cover that is also a partition. We prove that this happens if and only if $G$ is isomorphic to $C_p \times C_p$ for some prime $p$ or to a Frobenius group with Frobenius kernel being an abelian minimal normal subgroup and Frobenius complement cyclic.
\end{abstract}

\maketitle

\tableofcontents

\section{Introduction}

In this paper all groups considered will be finite. Let $G$ be a group. A cover of $G$ is a list of proper subgroups $H_1,\ldots,H_n$ of $G$ with the property that $H_1 \cup \ldots \cup H_n = G$. A cover is called a partition of $G$ if $H_i \cap H_j = \{1\}$ whenever $i \neq j$. Every non-cyclic group admits a cover but this is not true for partitions. A group is called partitionable if it admits a partition. Partitionable groups were classified by Baer, Kegel and Suzuki in 1961 (see \cite{zappa} for a good account of how the proof was completed).

\ 

Baer, Kegel and Suzuki proved that a group $G$ is partitionable if and only if it is isomorphic to one of:

\begin{enumerate}

\item $S_4$ (the symmetric group of degree $4$),

\item a $p$-group with $H_p(G) \neq G$, where $H_p(G) = \langle x \in G\ :\ x^p \neq 1 \rangle$,

\item a group of Hughes-Thompson type,

\item a Frobenius group,

\item $\PSL_2(p^m)$ with $p^m \geq 4$,

\item $\PGL_2(p^m)$ with $p^m \geq 5$ and $p$ odd,

\item the Suzuki group $Sz(2^{2m+1})$ where $m \geq 1$,

\end{enumerate}

for some prime $p$ and some $m \in \mathbb{N}$.

\ 

Following \cite{cohn} we use the notation $\sigma(G)$ for the minimal size of a cover of $G$, with the convention that $\sigma(G)=\infty$ if $G$ is cyclic. The invariant $\sigma(G)$ has received a lot of interest in recent years. In particular, several new and powerful techniques to study $\sigma(G)$ were introduced in \cite{cubo}. A new and interesting trend is the comparison of group cover types. The authors of \cite{irrcov} classify those groups $G$ in which all the inclusion-minimal covers (covers with no proper subcovers) are size-minimal (that is, have size $\sigma(G)$). Following \cite{sizemore}, if $G$ is partitionable we use the notation $\rho(G)$ for the minimal size of a partition of $G$. Clearly, $\sigma(G)$ is always at most $\rho(G)$. A partition is called normal if it is closed under conjugation; a lot of work has been done by Baer to characterize the normal partitions, being this one of the main ideas of the classification of partitionable groups, however observe that the partitions of a group $G$ of size $\rho(G)$ in general are not normal. T. Foguel and N. Sizemore in \cite{fs} and \cite{sizemore} have raised the interesting problem of studying when equality occurs, and this is precisely what we do in this paper. We are interested in classifying the partitionable groups $G$ with $\sigma(G)=\rho(G)$. Such equality is equivalent to the statement that $G$ admits a minimal cover which is also a partition. In this paper we essentially compare the two invariants $\sigma(G)$ and $\rho(G)$. Our main theorem is the following.

\begin{teor} \label{main}
Let $G$ be a partitionable group. Then $\sigma(G)=\rho(G)$ if and only if $G \cong C_p \times C_p$ for some prime $p$ or $G$ is a Frobenius group with Frobenius kernel being an abelian minimal normal subgroup and Frobenius complement cyclic.
\end{teor}

Let $q=p^n \geq 4$ be a power of the prime $p$, and consider the partitionable projective groups $G=\PSL_2(q)$ with $q \geq 4$, and $G=\PGL_2(q)$ with $q \geq 5$ odd. Bryce, Fedri and Serena \cite{bfs} showed that $\sigma(\PSL_2(5))=10$, $\sigma(\PSL_2(7))=15$, $\sigma(\PSL_2(9))=16$ and in all other cases $$\sigma(\PSL_2(q))=\sigma(\PGL_2(q))= \left \{ \begin{array}{cc} \frac{1}{2}q(q+1) & \mbox{if q is even,} \\ \frac{1}{2}q(q+1)+1 & \mbox{if q is odd.} \end{array} \right.$$ The following result, combined with the discussion above, implies that $\rho(G) > \sigma(G)$ for $G= \PSL_2(q)$, $\PGL_2(q)$.

\begin{teor} \label{linear}
Let $q \geq 3$ be a prime power. Then $\rho(\PGL_2(q))=q^2+1$, furthermore $\rho(\PSL_2(q))=q^2+1$ if $q \geq 7$.
\end{teor}

This is proved in Sections \ref{sects4} and \ref{sectlin}. Note that this includes all projective partitionable groups because $\PSL_2(q) \cong \PGL_2(q)$ if $q$ is even and $\PSL_2(5) \cong \PSL_2(4)$.

\ 

The paper is organized as follows. In Section \ref{general} we list some general properties of covers and partitions, and in the subsequent sections we deal with each individual family of the classification of partitionable groups by Baer, Kegel and Suzuki. In Section \ref{sectpgps} we also compute $\rho(G)$ when $G$ is elementary abelian (this was done in \cite{fs} but our techniques are different), and in Section \ref{sectht} we compute $\rho(G)$ when $G$ is a group of Hughes-Thompson type and not a Frobenius group. Suzuki groups are treated in Section \ref{sectsuz}.

\ 

An interesting question is the following: is it true that if $G$ is a Frobenius group then $\rho(G)=|K|+1$, where $K$ is the Frobenius kernel of $G$? Some work has been done in this direction, however we still do not have a definite answer.

\section{General remarks} \label{general}

Since any cover of a quotient $G/N$ of a group $G$ can be lifted to a cover of $G$ we have the basic inequality $\sigma(G) \leq \sigma(G/N)$ whenever $N \unlhd G$. Tomkinson in \cite{tom} computed $\sigma(G)$ for every solvable group $G$. Recall that a chief factor of $G$ is a minimal normal subgroup of a quotient of $G$, and a complement of a chief factor $L/N$ of $G$ (where $N \unlhd G$) is a subgroup $K/N$ of $G/N$ with the property that $KL=G$ and $K \cap L = N$.

\begin{teor}[Tomkinson \cite{tom}]
Let $G$ be a solvable group. Then $\sigma(G)=q+1$ where $q$ is the smallest order of a chief factor of $G$ with more than one complement.
\end{teor}

When $G$ is nilpotent this is saying that $\sigma(G) = p+1$ where $p$ is the smallest prime divisor of $|G|$ such that the Sylow $p$-subgroup of $G$ is noncyclic (this particular case can be deduced via more elementary arguments: see \cite{cohn}).

\begin{lemma} \label{upb}
Let $G$ be a partitionable group. Then $\rho(G) < |G|$.
\end{lemma}

\begin{proof}
Let $\mathscr{P} = \{H_1,\ldots,H_m\}$ be a partition of $G$ consisting of nontrivial subgroups. Then clearly $$|G| = 1 + \sum_{i=1}^m (|H_i|-1) = 1-m + \sum_{i=1}^m |H_i| \geq 1-m+2m = m+1.$$ Since any partition of size $\rho(G)$ consists of nontrivial subgroups it follows that $\rho(G) \leq |G|-1$.
\end{proof}

\begin{lemma} \label{trick}
Let $\mathscr{P} = \{H_1,\ldots,H_m\}$ be a partition of the group $G$ consisting of nontrivial subgroups. Then $|\mathscr{P}| \geq 1 + |H_i|$ for any $i=1,\ldots,m$, and if $H_i$ is normal in $G$ and $|G:H_i|$ is a prime then $|\mathscr{P}| = 1 + |H_i|$.
\end{lemma}

\begin{proof}
Fix $i \in \{1,\ldots,m\}$. Let $j \in \{1,\ldots,m\}$ with $j \neq i$. Since $H_i \cap H_j = \{1\}$ we have $|G| \geq |H_i H_j| = |H_i| |H_j|$ thus $|H_j| \leq |G:H_i|$ and we obtain
\begin{align*}
|G|-|H_i| & = \sum_{i \neq j=1}^m |H_j|-(m-1) \leq \sum_{i \neq j=1}^m |G:H_i|-(m-1) \\ & = (m-1)(|G:H_i|-1) = \frac{1}{|H_i|} (m-1)(|G|-|H_i|), 
\end{align*}
and this implies $|\mathscr{P}| = m \geq 1 + |H_i|$. Now suppose that $H_i$ is normal, and let $j \in \{1,\ldots,m\}$ with $j \neq i$. Then $H_i H_j \leq G$ thus $|H_iH_j| = |H_i||H_j|$ divides $|G|$, i.e. $|H_j|$ divides $|G:H_i|$. Therefore if $|G:H_i|$ is a prime then $|H_j| = |G:H_i|$ whenever $j \neq i$ and the above inequality becomes an equality.
\end{proof}

For $x$ a real number denote by $\left\lceil x \right\rceil$ the smallest integer $k$ such that $k \geq x$.

\begin{cor} \label{sqrt}
If $G$ is a partitionable group then $\rho(G) \geq 1 + \left\lceil \sqrt{|G|} \right\rceil$.
\end{cor}

\begin{proof}
It is clearly enough to show that $\rho(G) \geq 1 + \sqrt{|G|}$. Let $\mathscr{P}$ be a partition of $G$. If there exists $H \in \mathscr{P}$ with $|H| \geq \sqrt{|G|}$ then, by Lemma \ref{trick} $\rho(G) \geq 1+|H| \geq 1+\sqrt{|G|}$. Suppose now that $|H| < \sqrt{|G|}$ for all $H \in \mathscr{P}$. Then for $\mathscr{P} = \{H_1,\ldots,H_m\}$ we have $$|G| = 1-m + \sum_{i=1}^m |H_i| < 1-m + m \sqrt{|G|},$$ which implies $m > \sqrt{|G|}+1$.
\end{proof}

\begin{cor}
Let the partitionable group $G$ admit a normal cyclic subgroup $N$ of prime index. Then $\rho(G) = 1 + |N|$.
\end{cor}

\begin{proof}
Being cyclic and maximal, the subgroup $N$ of $G$ must belong to any partition. Now apply Lemma \ref{trick}.
\end{proof}

\begin{cor} \label{embed}
Let $H,G$ be partitionable groups with $H \leq G$. Then $\rho(H) \leq \rho(G)$.
\end{cor}

\begin{proof}
Let $\mathscr{P} = \{H_1,\ldots,H_m\}$ be a partition of $G$ of size $\rho(G)$. If there exists $i \in \{1,\ldots,m\}$ such that $H \subseteq H_i$ then Lemma \ref{upb} and Lemma \ref{trick} imply $\rho(G) \geq 1+|H_i| \geq 1+|H| > |H| > \rho(H)$. Now assume that $H \not \subseteq H_i$ for all $i \in \{1,\ldots,m\}$. Then $\{H_1 \cap H, \ldots, H_m \cap H\}$ is a partition of $H$ of size $m = \rho(G)$, hence $\rho(H) \leq \rho(G)$.
\end{proof}

\section{$p$-groups} \label{sectpgps}

In this section we consider $p$-groups. First, we characterize partitionable nilpotent groups $G$ with $\sigma(G)=\rho(G)$ as follows.

\begin{prop} \label{mainpgp}
Let $G$ be a partitionable nilpotent group. Then $\sigma(G) = \rho(G)$ if and only if $G \cong C_p \times C_p$ for some prime $p$.
\end{prop}

\begin{proof}
Of course $\sigma(C_p \times C_p) = p+1 = \rho(C_p \times C_p)$. Now assume $G$ is partitionable with $\sigma(G)=\rho(G)$. Tomkinson's theorem implies that $\sigma(G) = p+1$ where $p$ is the smallest prime divisor of $|G|$ such that the Sylow $p$-subgroup of $G$ is noncyclic, hence $\rho(G) = \sigma(G) = p+1$ and Corollary \ref{sqrt} implies $p+1 \geq \sqrt{|G|}+1$, hence $|G| \leq p^2$. Since the Sylow $p$-subgroup of $G$ is noncyclic it has size at least $p^2$, therefore $|G| = p^2$ and $G \cong C_p \times C_p$.
\end{proof}

We now want to compute $\rho(G)$ when $G$ is an elementary abelian $p$-group. Let $p$ be a prime and let $n$ be a positive integer larger than $1$. Let ${C_p}^n$ denote the elementary abelian $p$-group of rank $n$.

\begin{prop} \label{cpn} $\rho({C_p}^n) = 1+p^{\left\lceil n/2 \right\rceil}$. \end{prop}

\begin{proof} Suppose first that $n$ is even. Note that we can look at $G = C_p^n$ as a vector space of dimension $2$ over the field $F$ of size $p^{n/2}$. Two distinct $1$-dimensional $F$-subspaces of $G$ intersect trivially, and clearly $G$ is the union of its $1$-dimensional $F$-subspaces. It follows that $$\rho(G) \leq \frac{|G|-1}{|F|-1} = 1 + |F| = 1 + p^{n/2} = 1+\sqrt{|G|}.$$ Corollary \ref{sqrt} then implies that $\rho(G) = 1 + p^{n/2}$.

\ 

Now assume that $n$ is odd. We need to show that $m = \rho({C_p}^n)$ equals $1 + p^{(n+1)/2}$. Note that ${C_p}^n$ embeds in ${C_p}^{n+1}$ as a normal subgroup of index $p$. Corollary \ref{embed} implies that $m \leq \rho({C_p}^{n+1}) = 1 + p^{(n+1)/2}$, our upper bound. Consider now a partition $\mathscr{P} = \{H_1,\ldots,H_m\}$ of ${C_p}^n$ of size $m$. We have $1 + p^{(n+1)/2} \geq m \geq 1 + |H_i|$ by Lemma \ref{trick} thus $|H_i| \leq p^{(n+1)/2}$, for all $i \in \{1,\ldots,m\}$. If $|H_i| = p^{(n+1)/2}$ for some $i \in \{1,\ldots,m\}$ then we deduce $m = 1 + p^{(n+1)/2}$. Now assume that $|H_i| \leq p^{(n-1)/2}$ for all $i \in \{1,\ldots,m\}$. Then we have $$p^n = 1-m + \sum_{i=1}^m |H_i| \leq 1-m + m p^{(n-1)/2} \leq 1+(p^{(n+1)/2}+1)(p^{(n-1)/2}-1),$$a contradiction.
\end{proof}

\begin{cor}
If $G$ is any partitionable group then $\rho(G) \leq |G|-1$ and equality holds if and only if $G \cong C_2 \times C_2$.
\end{cor}

\begin{proof}
The inequality $\rho(G) \leq |G|-1$ follows from Lemma \ref{upb}. If equality holds then in the proof of Lemma \ref{upb} we have $|H_i|=2$ for all $i=1,\ldots,m$ so $G$ is an elementary abelian $2$-group and the result follows directly from Proposition \ref{cpn}.
\end{proof}

\section{The symmetric group $S_4$} \label{sects4}

Recalling that $\sigma(S_4)=4$ (see \cite{cohn}) the following result implies that $\sigma(S_4) \neq \rho(S_4)$. It was obtained independently by N. Sizemore in \cite{sizemore} using GAP. Observe that $S_4 \cong \PGL_2(3)$ hence this proves Theorem \ref{linear} in this particular case.

\begin{prop}
$\rho(S_4)=10$.
\end{prop}

\begin{proof}
We will think of the group $G=S_4$ acting naturally on $\{1,2,3,4\}$. $G$ has $9$ elements of order $2$, $8$ elements of order $3$ and $6$ elements of order $4$. The maximal subgroups of $S_4$ are of one of the following types: $S_3$ (the point stabilizers), $A_4$ (the alternating group of degree $4$), $D_8$ (the Sylow $2$-subgroups). Let $\mathscr{P}$ be a partition of $S_4$ of size $\rho(S_4)$. We will show that $|\mathscr{P}|=10$. Observe that $A_4 \not \in \mathscr{P}$ because if $A_4 \in \mathscr{P}$ then for all $H \in \mathscr{P}$ distinct from $A_4$ we have $|H| \leq |S_4:A_4|=2$ and this is impossible because $S_4$ has elements of order $4$ outside $A_4$. Observe that $\mathscr{P}$ cannot contain a member $D$ isomorphic to $D_8$ because if $D_8 \cong D \in \mathscr{P}$ and $H \in \mathscr{P}-\{D\}$ then $|H| \leq |S_4:D|=3$, this is impossible because $D_8$ contains only two elements of order $4$ and $S_4$ contains $6$ elements of order $4$. This implies that the three cyclic subgroups of order $4$ must belong to $\mathscr{P}$ (the only proper subgroups properly containing them are the dihedral subgroups of order $8$). Now assume $S_3 \cong S \in \mathscr{P}$. Then if $H \in \mathscr{P}-\{S\}$ we have $|H| \leq |S_4:S| = 4$ so $S$ is the unique member of $\mathscr{P}$ isomorphic to $S_3$. In any case we need at least four subgroups to cover the elements of order $3$. There are three $2$-cycles left uncovered so $|\mathscr{P}|=3+4+3=10$.
\end{proof}

\section{Frobenius groups} \label{sectfrob}

In this section we prove Theorem \ref{main} when $G$ is a Frobenius group.

\ 

We recall a known fact due to Gasch\"utz, Carter, Fischer, Hawkes, and Barnes. It basically says that the number of complements of the chief factors of a given chief series of a solvable group $G$ does not depend, up to a permutation, on the given chief series. In particular if no chief factor of a chief series of $G$ has a given number of complements $n$ then no chief factor of $G$ has precisely $n$ complements. It is a more or less straightforward consequence of \cite[Theorem 9.13 on page 33]{dh} and \cite[Satz 3]{gasch}. A proof can be found in \cite{barnes}. We will use it afterwards without mention.

\begin{prop}
If $G$ is a solvable group and $\{1\} = K_0 < K_1 < \ldots < K_t=G$ is a chief series of $G$, with $n_i$ the number of complements of $K_i/K_{i-1}$ in $G/K_{i-1}$ for all $i=1,\ldots,t$ (so that for example $n_i=0$ if and only if $K_i/K_{i-1}$ is Frattini), the sequence of numbers $(n_1,\ldots,n_t)$, up to reordering, does not depend on the given chief series.
\end{prop}

The following result is well-known, see for example \cite[Lemmas 4.20 and 8.4]{dh}.

\begin{lemma} \label{mnmax}
Let $G$ be a solvable group and let $N$ be a minimal normal subgroup of $G$. Then $N$ is an elementary abelian $p$-group for some prime $p$ and if $N$ is complemented then the complements of $N$ in $G$ are precisely the maximal subgroups of $G$ that do not contain $N$.
\end{lemma}

Before considering Frobenius groups we need a lemma.

\begin{lemma} \label{cfcyc}
A finite solvable group $G$ is cyclic if and only if every chief factor of $G$ has $0$ or $1$ complements.
\end{lemma}

\begin{proof}
If $G$ is cyclic then it is a direct product $C_{{p_1}^{l_1}} \times \cdots \times C_{{p_t}^{l_t}}$, the chief factors of $G$ are the chief factors of its direct factors, which are cyclic groups of prime power order. But cyclic $p$-groups only have one chief series, all of its chief factors have $0$ complements except the top one, that has $1$ complement.

We show the converse by induction on $|G|$. Denote by $\Phi(G)$ the Frattini subgroup of $G$. If $G/\Phi(G)$ is cyclic then $G$ is cyclic, and the chief factors of $G/\Phi(G)$ are in particular chief factors of $G$, so we may assume $\Phi(G)=\{1\}$. Let $N$ be a minimal normal subgroup of $G$. Then $N$ has precisely one complement: if $N$ had no complement then $N \subseteq \Phi(G)$ which cannot be since $\Phi(G)=\{1\}$. Call $L$ the complement of $N$; since the conjugates of $L$ in $G$ are also complements of $N$ it follows that $L$ is normal in $G$ and $G \cong N \times L$. Since $G$ is solvable, $N$ is abelian by Lemma \ref{mnmax}, hence $N$ is a central minimal normal subgroup of $G$ so $N \cong C_p$ for some prime $p$. By induction, $L$ is cyclic. Since $G \cong C_p \times L$, to conclude that $G$ is cyclic it is enough to show that $p$ does not divide $|L|$. If this was not the case we could write $G = C_p \times C_{p^l} \times H$ for some $l \geq 1$ and for some cyclic subgroup $H$, so $G$ would have a quotient isomorphic to $C_p \times C_p$. However in this group the chief factor $C_p \times \{1\}$ has $p$ complements, so more than one, a contradiction.
\end{proof}

Recall that a Frobenius group is a group $G$ with a proper subgroup $H$ with the property that $H \cap H^g = \{1\}$ whenever $g \in G-H$, so that $H$ is self-normalizing in $G$. By a celebrated theorem by Frobenius, there exists a normal subgroup $K$ in $G$ such that $KH=G$ and $K \cap H = \{1\}$, making $G$ a semidirect product $K \rtimes H$. Moreover $K$ and $H$ have coprime orders and $K$ equals the Fitting subgroup of $G$, in particular it is nilpotent. The subgroup $K$ is called the Frobenius kernel of $G$ and the subgroup $H$ is called a Frobenius complement of $G$. All the complements of $K$ in $G$ are conjugate. For other basic properties of Frobenius groups we refer to \cite[Chapter 37]{kar}.

\begin{prop} \label{mainfrob}
Let $G$ be a Frobenius group. The following are equivalent.
\begin{enumerate}
\item $\rho(G)=\sigma(G)$.
\item The Frobenius kernel is a minimal normal subgroup of $G$ and the Frobenius complements are cyclic.
\end{enumerate}
\end{prop}

\begin{proof}
Consider a Frobenius group $G=K \rtimes H$ where $K$ is the Frobenius kernel and $H$ is a Frobenius complement. Then $K \cap H^g = \{1\}$ for all $g \in G$ and $H^x \cap H^y = \{1\}$ whenever $xy^{-1} \not \in H$. In particular the union $K \cup \bigcup_{g \in G}H^g$ covers exactly $$|K|+|K|(|H|-1) = |G|$$ elements. Therefore such a union is a partition of $G$ with $|K|+1$ members (this of course is well-known).

\ 

Assume that $\sigma(G)=\rho(G)$. We claim that $G$ is solvable. To prove this assume by contradiction that $G$ is non-solvable. Since $K$ is nilpotent, $H$ is not solvable. By a theorem of Zassenhaus, who classified nonsolvable Frobenius complements (see \cite[Theorem 18.6]{passman}), $H$ must have a subgroup $H_0$ of index $1$ or $2$ such that $H_0 \cong S \times M$ with $S \cong \SL_2(5)$ and $M$ a group of order coprime to $|S|=120$. Since $(|S|,|M|) = 1$, $M$ is characteristic in $H_0$ hence normal in $H$, so $\sigma(G) \leq \sigma(H) \leq \sigma(H/M)$ and $H/M$ is either isomorphic to $S$ or to an extension $S.2$. Since $S/Z(S) \cong A_5$, $H/M$ has a quotient group isomorphic to $A_5$ or $S_5$. Using Corollary \ref{sqrt} we deduce that $$1+\sqrt{|G|} \leq \rho(G) = \sigma(G) \leq \max \{ \sigma(A_5),\sigma(S_5)\} = \sigma(S_5) = 16$$ hence $|G| \leq 225$. But since $120=|\SL_2(5)|$ divides $|G|$ it follows that $|G|=120$ hence $G=\SL_2(5)$, and this contradicts the fact that $G$ is a Frobenius group.

\ 

We deduced that $G$ is solvable, and it follows from Tomkinson's Theorem that $\sigma(G)=q+1$ where $q=p^r$ is the size of the smallest chief factor of $G$ with more than one complement. Therefore $\rho(G)=q+1$ and if $U_1 \cup U_2 \cup \ldots \cup U_t = G$ is a partition with $t=\rho(G)$ members, then, by Lemma \ref{trick}, we have $1+q = \rho(G) \geq 1+|U_i|$ hence $|U_i| \leq q$ for all $i=1,\ldots,t$, and by Corollary \ref{sqrt} we have $1+\sqrt{|G|} \leq \rho(G) = 1+q$ so $|G| \leq q^2$. Consider a fixed chief series of $G$. Since $|G| \leq q^2$ and $G$ is not a $p$-group (being $G$ a Frobenius group) it is clear that precisely one chief factor in the series has more than one complement. Therefore if a chief factor of $G$ is not a $p$-group then it is either Frattini or central, implying that $G$ is $\ell$-nilpotent for all prime divisor $\ell$ of $|G|$ different from $p$ (see for example \cite[Lemma 5]{irrcov}). Let $P$ be the intersection of the normal $\ell$-complements of $G$ for $\ell$ a prime divisor of $|G|$ different from $p$. Then $P$ is the unique Sylow $p$-subgroup of $G$ and the Schur-Zassenhaus theorem implies that $G = P \rtimes T$ for some $T \leq G$. Since $q$ is a power of $p$ and $p$ does not divide $|T|$, Lemma \ref{cfcyc} implies that $T$ is cyclic. Since $P$ is a normal nilpotent subgroup of $G$, $P$ is contained in the Fitting subgroup of $G$, which equals the Frobenius kernel $K$. If $P \neq K$ then $G/P$ is a Frobenius group, contradicting the fact that $G/P \cong T$ is cyclic. So $P=K$, $T$ is conjugate to $H$ and all the Frobenius complements are cyclic.

\ 

Let $A/B$ be the unique chief factor with more than one complement in a chief series of $G$ passing through $P$, so that $|A/B|=q$. We claim that $A=P$ and $B = \{1\}$, implying the result. The group $G/A$ is cyclic by Lemma \ref{cfcyc}, so $A$ cannot be properly contained in $P$ (because otherwise $G/A$ would be a Frobenius group), on the other hand $A/B$ is a $p$-group (it has order $q$) so $A=P$. We now prove that $B$ is contained in the Frattini subgroup of $G$, $\Phi(G)$. Let $\{1\} = L_1 < L_2 < \ldots < L_r = B$ be part of a chief series of $G$ passing through $B$. Since $G/L_i$ is a Frobenius group for $i=1,\ldots,r$ we have $Z(G/L_i)=\{1\}$ for $i=1,\ldots,r$ so the chief factors $L_i/L_{i+1}$ are all Frattini (they all have $0$ complements), indeed if one of them was not Frattini it would have exactly one complement hence it would be central. If $M$ is a maximal subgroup of $G$ then it contains $L_2 \leq \Phi(G)$ hence by induction $M/L_i$ contains $L_{i+1}/L_i$ for all $i=1,\ldots,r-1$, implying that $M \geq L_r = B$. This proves that $B \leq \Phi(G)$. We know that $\rho(G)=\sigma(G)=q+1$, let $\{H_1,\ldots,H_{q+1}\}$ be a partition of $G$. Since $B \leq \Phi(G)$ we have $H_iB \neq G$ for all $i=1,\ldots,q+1$, so the family $\{H_1B/B,\ldots,H_{q+1}B/B\}$ is a cover of $G/B$, and the $q$ Frobenius complements of $G/B$ are cyclic maximal subgroups of the form $\langle x \rangle B/B$ where $x$ generates a Frobenius complement of $G$. It follows that up to reordering $H_iB=\langle x_i \rangle B$ for $i=1,\ldots,q$ where $x_i$ generates a Frobenius complement of $G$ for $i=1,\ldots,q$, in particular $B$ is the unique Sylow $p$-subgroup of $H_iB$, and this implies $P = G \cap P = \bigcup_{i=1}^{q+1} (H_i \cap P) \subseteq B \cup H$ where $H=H_{q+1}$. Since no group is the union of two proper subgroups, we deduce that $P \leq H$, so, by Lemma \ref{trick}, $1+q =\rho(G) \geq 1+|H| \geq 1+|P|$ implying $|P| \leq q=|P/B|$ so $B=\{1\}$.

\ 

Conversely, assume that the Frobenius kernel $K$ is a minimal normal subgroup of $G$ and that the Frobenius complements are cyclic. We prove that $\sigma(G)=\rho(G)$. Observe that $H$ is a maximal subgroup of $G$ by Lemma \ref{mnmax}, so all the Frobenius complements are maximal subgroups of $G$. Let $\mathscr{P}$ be any partition of $G$ of size $\rho(G)$. If $\langle x_1 \rangle, \ldots, \langle x_t \rangle$ are the Frobenius complements then for each $x_i$ exactly one member of $\mathscr{P}$ contains $x_i$, therefore it is equal to $\langle x_i \rangle$ since such subgroup is maximal. On the other hand $(\langle x_1 \rangle \cup \ldots \cup \langle x_t \rangle) \cap K = \{1\}$, therefore $\rho(G) = |\mathscr{P}| \geq t+1 = |K|+1$. Since a partition of $G$ is given by the Frobenius kernel and the Frobenius complements it follows that $\rho(G)=|K|+1$. Since $G/K$ is cyclic, $K$ is the unique chief factor of $G$ with more than one complement in a chief series of $G$ passing through $K$, so Tomkinson's theorem implies $\sigma(G) = |K|+1 = \rho(G)$.
\end{proof}

\section{Groups of Hughes-Thompson type} \label{sectht}

A group of Hughes-Thompson type relative to the prime $p$ is a group $G$ such that $G$ is not a $p$-group and the subgroup generated by the elements of order different from $p$ $$H_p(G) = \langle x \in G\ :\ x^p \neq 1 \rangle$$is not equal to $G$. For such a group $|G:H_p(G)| = p$ and $H_p(G)$ is a normal nilpotent subgroup of $G$ (see \cite{bryce}). Observe that the class of groups of Hughes-Thompson type is not contained in the class of Frobenius groups, consider as an example the dihedral group of order $12$.

\begin{prop}
Let $G$ be a group of Hughes-Thompson type relative to the prime $p$. Then we have the following.
\begin{enumerate}
\item If $G$ is not a Frobenius group then $H=H_p(G)$ is a member of every partition of $G$ and $\rho(G)=|H|+1$.
\item $\sigma(G)=\rho(G)$ if and only if $G$ is a Frobenius group with Frobenius kernel a minimal normal subgroup and Frobenius complement of order $p$.
\end{enumerate}
\end{prop}

\begin{proof}
(1) The statement about $\rho(G)$ follows from Lemma \ref{trick}. Let $\mathscr{P}$ be a partition of $G$ and assume $H$ does not appear in $\mathscr{P}$. Intersecting the members of $\mathscr{P}$ with $H$ we find a partition of $H$, so $H$ is a partitionable nilpotent group, and by the classification of partitionable groups it follows that there exists a prime $r$ such that $H$ is an $r$-group, and $r \neq p$ because $G$ is not a $p$-group. Calling $P$ a Sylow $p$-subgroup of $G$ we have $G=H \rtimes P$. If $x \in P$ and $1 \neq y \in H$ are such that $xy=yx$ then $(xy)^p=y^p \neq 1$ being $|H|$ coprime to $p$, hence $xy \in H_p(G) = H$ by definition of $H_p(G)$, implying $x \in H \cap P$, so $x=1$. This proves that the centralizer $C_H(x)$ is trivial whenever $x \in P-\{1\}$, therefore $G$ is a Frobenius group with kernel $H$ and complement $P$.

(2) Since groups of order $p$ are cyclic, the implication $\Leftarrow$ is immediate from Proposition \ref{mainfrob}, now we prove the converse. Let $G$ be a group of Hughes-Thompson type relative to the prime $p$, and not a $p$-group, with the property that $\sigma(G)=\rho(G)$. If $G$ is a Frobenius group the result follows from Proposition \ref{mainfrob}, now assume $G$ is not a Frobenius group. Let $\mathscr{P}$ be a partition of $G$ of cardinality $\rho(G)$, then by (1) $H=H_p(G)$ belongs to $\mathscr{P}$ hence Lemma \ref{trick} implies $1+|H|=\rho(G)=\sigma(G)=1+q$ where $q$ is the unique smallest size of a chief factor of $G$ with more than one complement. It follows that $H$ is a minimal normal subgroup of $G$ of order $q$ and index $p$, hence $G$ is a Frobenius group, a contradiction.
\end{proof}

\section{$\PSL_2(q)$ and $\PGL_2(q)$} \label{sectlin}

In this section we prove Theorem \ref{linear}. We use Dickson's classification of the maximal subgroups of $\PSL_2(q)$ and $\PGL_2(q)$. It can be found in \cite{dickson}. Observe that if $g \in \GL_2(q)$ has characteristic polynomial $f(X)$ which is reducible then $g$ is contained in the stabilizer of a point of the projective line, which is a Frobenius group with Frobenius kernel elementary abelian of order $q$, and if $f(X)$ is irreducible then $\mathbb{F}_q[g]$ is a field of order $q^2$ so $g$ lies in a Singer cycle, that is, a cyclic subgroup of order $q^2-1$ corresponding to an element of $F=\mathbb{F}_{q^2}$ acting on $F$ by multiplication. This implies that the point stabilizers together with the Singer cycles form a cover of $\GL_2(q)$ and induce covers of $\PGL_2(q)$ and of $\PSL_2(q)$. Observe that the Singer cycles have order $q+1$ in $\PSL_2(q)$ when $q$ is even, $q+1$ in $\PGL_2(q)$ when $q$ is odd and $(q+1)/2$ in $\PSL_2(q)$ when $q$ is odd. Also, the family consisting of Singer cycles and the cyclic subgroups generated by the maximal semisimple (diagonalizable) elements consists of subgroups that have pairwise trivial intersection.

\ 

Suppose $q=2^f$. Let $G=\PSL_2(q) \cong \PGL_2(q) \cong \SL_2(q)$. The group $G$ has order $q(q^2-1)$, exponent $2(q^2-1)$, and it contains $q^2-1$ involutions. The maximal subgroups of $G$ are:
\begin{itemize}
\item ${C_2}^f \rtimes C_{q-1}$ (stabilizer of a point of the projective line), one conjugacy class - we will call such subgroup ``point stabilizer'';
\item $D_{2(q-1)}$ (normalizer of a cyclic subgroup generated by a semisimple element of order $q-1$), one conjugacy class;
\item $D_{2(q+1)}$ (Singer cycle normalizer), one conjugacy class;
\item $\PGL_2(q_0)$ where $q=q_0^r$ for some prime $r$ and $q_0 > 2$.
\end{itemize}

Before proceeding to the computation of $\rho(G)$ we need two observations.

\ 

We claim that there is a conjugate of $D=D_{2(q-1)}$ distinct from $D$ and intersecting $D$ non-trivially. We think of $G$ as of $\SL_2(q)$. Let $\alpha$ be a generator of the multiplicative group $\mathbb{F}_q^{\ast}$, and without loss of generality $$D = \langle \left( \begin{array}{cc} \alpha & 0 \\ 0 & \alpha^{-1} \end{array} \right),\ \left( \begin{array}{cc} 0 & 1 \\ 1 & 0 \end{array} \right) \rangle = \bigcup_{m=0}^{q-2} \left\{ \left( \begin{array}{cc} \alpha^m & 0 \\ 0 & \alpha^{-m} \end{array} \right),\ \left( \begin{array}{cc} 0 & \alpha^m \\ \alpha^{-m} & 0 \end{array} \right) \right\}.$$Let $t$ be an element of $\mathbb{F}_q$ different from $0$ and $1$, and let $g := \left( \begin{array}{cc} t & t+1 \\ 1 & 1 \end{array} \right) \in G$. Since $|\mathbb{F}_q^{\ast}|=q-1$ is odd there exists $s \in \mathbb{F}_q$ with $s^2=t(t+1)$. The conjugate $$g^{-1} \left( \begin{array}{cc} 0 & s \\ s^{-1} & 0 \end{array} \right) g = \left( \begin{array}{cc} 0 & st^{-1} \\ s^{-1}t & 0 \end{array} \right)$$is an involution belonging to $D \cap D^g$, however $D \neq D^g$ being $g \not \in D=N_G(D)$.

\ 

We claim that the intersection of any two maximal subgroups of type $D_{2(q+1)}$ has order $2$. Let $D$ be a maximal subgroup of $G$ of type $D_{2(q+1)}$. First observe that, since any two Singer cycles have trivial intersection, the intersection between any two conjugates of $D$ has order at most $2$. Let $d$ be the number of pairs $(H,x)$ where $H$ is a conjugate of $D$ and $x$ is an involution belonging to $H$. Since $D$ has $q(q-1)/2$ conjugates and each of them contains $q+1$ involutions, $d=(q+1)q(q-1)/2=|G|/2$. On the other, hand all the involutions of $G$ are pairwise conjugate, so each of them belongs to a constant number $c$ of conjugates of $D$, hence $|G|/2=d=c(q^2-1)$, so $c=q/2$. This means that every involution belongs to $q/2$ conjugates of $D$. Each of the $q+1$ involutions in $D$ belongs to $q/2-1$ conjugates of $D$ distinct from $D$, and no two involutions of $D$ belong to the same conjugate of $D$ distinct from $D$ because the intersection of $D$ with its distinct conjugates has order at most $2$. This means that $D$ has non-trivial intersection with $(q+1)(q/2-1)+1 = (q^2-q)/2$ of its conjugates, hence with all of its conjugates.

\ 

Now we prove that $\rho(G)=q^2+1$. A direct computation shows that any two Sylow $2$-subgroups of $G$ intersect trivially. A partition of $G$ consists of a point stabilizer $H$, all the cyclic subgroups of order $q-1$ not contained in $H$, all the cyclic subgroups of order $q+1$ and the $q$ Sylow $2$-subgroups not contained in $H$. Since the cyclic subgroups of order $q-1$ have normalizer of order $2(q-1)$ and the cyclic subgroups of order $q+1$ have normalizer of order $2(q+1)$ we deduce that $$\rho(G) \leq 1+\left( \frac{|G|}{2(q-1)}-q \right)+ \frac{|G|}{2(q+1)} +q = q^2+1.$$ We are left to prove that $\rho(G) \geq q^2+1$. Let $\mathscr{P}$ be a partition of $G$. We will show that $|\mathscr{P}| \geq q^2+1$.

If $\mathscr{P}$ contains some point stabilizer $H \cong {C_2}^f \rtimes C_{q-1}$ then any other member of the partition has order at most $|G:H|=q+1$ (see the proof of Lemma \ref{trick}). Since $2(q-1) > q+1$, and if $q={q_0}^r$ with $q_0 > 2$ and $r$ a prime then $\PGL_2(q_0)$ does not contain elements of order $q-1$ nor $q+1$, as a simple examination of the list of maximal subgroups shows, it follows that $\mathscr{P}$ must contain all of the cyclic subgroups of order $q-1$ not contained in $H$ and all of the cyclic subgroups of order $q+1$ (Singer cycles), which are all conjugate, this gives a total of $|G|/(2(q-1))-q+|G|/(2(q+1)) = q^2-q$ subgroups. Note that $G$ contains $q^2-1$ involutions and every proper subgroup of $G$ contains at most $q+1$ involutions, however any proper subgroup of $G$ that is not of the type $D_{2(q+1)}$ contains at most $q-1$ involutions. Call $n$ the number of subgroups in $\mathscr{P}$ containing involutions. Since any two $D_{2(q+1)}$ intersect non-trivially $\mathscr{P}$ contains at most one of them, so we must have $q+1+(n-1)(q-1) \geq q^2-1$ which implies $n \geq q+1$ (recalling that $q \geq 4$). Since $H$ contains involutions this implies that we need at least $q$ additional subgroups to cover the Sylow $2$-subgroups and together with $H$ this gives $|\mathscr{P}| \geq 1+q^2-q+q = q^2+1$.

If $\mathscr{P}$ does not contain any point stabilizer then to cover the elements of order $q-1$ and $q+1$ the partition $\mathscr{P}$ must contain all of the cyclic subgroups of order $q-1$ (or their normalizers), which are maximal in the point stabilizers they are contained in, and all of the Singer cycles (or their normalizers), this gives $|G|/(2(q+1))+|G|/(2(q-1))=q^2$ subgroups. Let $a$ be the number of dihedral groups $D_{2(q-1)}$ in $\mathscr{P}$, and let $b$ be the number of dihedral groups $D_{2(q+1)}$ in $\mathscr{P}$. If $\mathscr{P}$ equals the family of the $q^2$ subgroups listed above then we must have $a(q-1) + b(q+1) = q^2-1$. However since there are two $D_{2(q-1)}$ intersecting non-trivially we have $a \leq q$, and since any two $D_{2(q+1)}$ intersect non-trivially we have $b \leq 1$. If $b=0$ then $q^2-1 = a(q-1) \leq q(q-1)$, a contradiction. If $b=1$ then $a(q-1)+q+1=q^2-1$ implying $a=(q^2-q-2)/(q-1)$ contradicting the fact that $a$ is an integer, being $q \geq 4$. This implies that $|\mathscr{P}| \geq q^2+1$.

\ 

Let $G=\PGL_2(q)$ with $q=p^f \geq 5$ and $p$ an odd prime. The group $G$ has order $q(q^2-1)$ and exponent $p(q^2-1)/2$, and it contains $q^2-1$ elements of order $p$, all contained in its socle $\PSL_2(q)$. The maximal subgroups of $G$ not containing $\PSL_2(q)$ are

\begin{itemize}
\item ${C_p}^f \rtimes C_{q-1}$ (stabilizer of a point of the projective line), one conjugacy class - we will call such subgroup ``point stabilizer'';
\item $D_{2(q-1)}$ for $q \neq 5$ (normalizer of a cyclic subgroup generated by a semisimple element of order $q-1$), one conjugacy class;
\item $D_{2(q+1)}$ (Singer cycle normalizer), one conjugacy class;
\item $S_4$ for $3 < q=p \equiv \pm 3 \mod 8$;
\item $\PGL_2(q_0)$ for $q=q_0^r$ with $r$ a prime.
\end{itemize}

Consider a point stabilizer $H$, the $q(q+1)/2-q$ cyclic subgroups of order $q-1$ not contained in $H$, the $q(q-1)/2$ cyclic subgroups of order $q+1$ and the $q$ Sylow $p$-subgroups not contained in $H$. Such subgroups form a partition of $G$ hence $\rho(G) \leq q^2+1$. Since $\PSL_2(q)$ is a subgroup of $\PGL_2(q)$, and they are both partitionable, we deduce that $\rho(\PSL_2(q)) \leq \rho(\PGL_2(q)) \leq q^2+1$ hence if we prove that $q^2+1 \leq \rho(\PSL_2(q))$ we can deduce immediately that $\rho(\PSL_2(q))=\rho(\PGL_2(q))=q^2+1$, which will be true for $q \neq 5$. In the following we will prove the lower bound $q^2+1 \leq \rho(\PSL_2(q))$ for $q \geq 13$ and we will treat the cases $q=5$, $q=7$, $q=9$, $q=11$ separately.

\ 

Let $G=\PSL_2(q)$ with $q=p^f \geq 5$ and $p$ an odd prime. $G$ has order $q(q^2-1)/2$ and exponent $p(q^2-1)/4$. Its maximal subgroups are

\begin{itemize}
\item ${C_p}^f \rtimes C_{(q-1)/2}$ (stabilizer of a point of the projective line), one conjugacy class - we will call such subgroup ``point stabilizer'';
\item $D_{q-1}$ for $q \geq 13$ (normalizer of a cyclic subgroup generated by a semisimple element of order $q-1$), one conjugacy class;
\item $D_{q+1}$ for $q \neq 7,9$ (Singer cycle normalizer), one conjugacy class;
\item $\PGL_2(q_0)$ for $q={q_0}^2$;
\item $\PSL_2(q_0)$ for $q={q_0}^r$ where $r$ is an odd prime;
\item $A_4$, $S_4$ or $A_5$ for some values of $q$.
\end{itemize}

Suppose $q \geq 13$. We will prove that $q^2+1$ is a lower bound for $\rho(G)$ by working with the elements of order $(q-1)/2$, $(q+1)/2$ and $p$. Observe that $(q-1)/2 \geq 6$ being $q \geq 13$, so the subgroups $A_5$, $A_4$, $S_4$ do not contain elements of order $(q-1)/2$ nor $(q+1)/2$, and the same is true for the subgroups $\PGL_2(q_0)$, $\PSL_2(q_0)$ by simple examination of their maximal subgroups. Let $\mathscr{P}$ be a partition of $G$.

If $\mathscr{P}$ contains a point stabilizer $H={C_p}^f \rtimes C_{(q-1)/2}$ then any other subgroup in $\mathscr{P}$ has order at most $|G:H|=q+1$ (see the proof of Lemma \ref{trick}) so to cover the elements of order $(q-1)/2$ and $(q+1)/2$ we need the cyclic subgroups of order $(q-1)/2$ not contained in $H$ or their normalizers $D_{q-1}$ and the cyclic subgroups of order $(q+1)/2$ or their normalizers $D_{q+1}$. Observe that since $G$ contains $q^2-1$ elements of order $p$ and of the subgroups listed so far only $H$ contains elements of order $p$, precisely $q-1$ of them, we need $q$ additional subgroups to cover them: this is because being $q \geq 13$ every proper subgroup of $G$ contains at most $q-1$ elements of order $p$, including $A_4$, $S_4$ and $A_5$. Since $H$ contains $q$ cyclic subgroups of order $(q-1)/2$ this gives $|\mathscr{P}| \geq 1+q(q+1)/2-q+q(q-1)/2+q = q^2+1$.

If $\mathscr{P}$ does not contain any point stabilizer then to cover the elements of order $(q-1)/2$ and $(q+1)/2$ the partition $\mathscr{P}$ must contain all of the cyclic subgroups of order $(q-1)/2$, $(q+1)/2$ or their normalizers, this gives $q(q+1)/2+q(q-1)/2=q^2$ subgroups and they do not form a cover because they do not contain any element of order $p$, so $|\mathscr{P}| \geq q^2+1$ in this case too.

\ 

Before dealing with the exceptions we observe that if $\mathscr{P}$ is a partition of a group $G$ and $H \in \mathscr{P}$, $x \in G$ are such that some power $x^n$ is not $1$ and belongs to $H$ then $x \in H$. Indeed if $x$ did not belong to $H$ then $H$ would have nontrivial intersection with the member of $\mathscr{P}$ containing $x$. We will use this fact in the following discussions without mention. Also, we will use repeatedly that if $H,K$ are two distinct subgroups in a partition of $G$ then $|K| \leq |G:H|$ (see the proof of Lemma \ref{trick}), in particular a subgroup $H$ of $G$ cannot belong to any partition if $G$ contains elements of order larger than $|G:H|$ and not contained in $H$.

\begin{itemize}
\item $\rho(\PGL_2(5)) = 26 = 5^2+1$. We will think of $\PGL_2(5)$ as of the symmetric group $S_5$. The group $G$ has $25$ elements of order $2$, $20$ elements of order $3$, $30$ elements of order $4$, $24$ elements of order $5$ and $20$ elements of order $6$. The maximal subgroups of $G$ are of one of the following types: $A_5$ (alternating group of degree $5$), $S_4$ (point stabilizer), $S_3 \times S_2$ (maximal intransitive subgroup of type $(3,2)$) and $C_5 \rtimes C_4$ (Sylow $5$-subgroup normalizer, a Frobenius group of order $20$). Let $\mathscr{P}$ be a partition of $S_5$. We know that if $H,K \in \mathscr{P}$ are distinct then $|H| \leq |G:K|$, so $A_5$ and the point stabilizers $S_4$ do not belong to $\mathscr{P}$: $A_5 \not \in \mathscr{P}$ because of the $4$-cycles, and $S_4 \not \in \mathscr{P}$ because of the elements of order $6$. Suppose that $\mathscr{P}$ contains the normalizer of a Sylow $5$-subgroup, which is a maximal subgroup $H$ of type $C_5 \rtimes C_4$ and contains $10$ elements of order $4$. Since its order is $20$, and $20^2 > |S_5|$, every other member of $\mathscr{P}$ is not isomorphic to it, and it is not a point stabilizer as observed above, so it contains at most two $4$-cycles. This proves that the number of members of $\mathscr{P}$ containing $4$-cycles is $1+20/2=11$ or $30/2=15$ according to whether $\mathscr{P}$ contains $H$ or not. No one of them contains elements of order $6$, and any proper subgroup of $S_5$ contains at most two elements of order $6$, so we need at least $20/2=10$ of them. No one of the subgroups listed so far, except $H$, contains elements of order $5$, however any proper subgroup of $S_5$ different from $A_5$ contains at most $4$ elements of order $5$ so we need at least $24/4-1=5$ additional subgroups in the first case, and $24/4=6$ in the second case, giving $|\mathscr{P}| \geq 11+10+5=26$ in the first case and $|\mathscr{P}| \geq 15+10+6=31$ in the second case. This proves that $\rho(G) \geq 26$. On the other hand a partition of $G$ is given by one $C_5 \rtimes C_4$ (call it $H$), the $5$ $C_5$ not contained in $H$, the $10$ $C_4$ not contained in $H$ and the $10$ $C_6$. This gives the upper bound $\rho(G) \leq 26$ and proves the result.

\ 

\item $\rho(\PSL_2(7)) = 50 = 7^2+1$. Let $G=\PSL_2(7)$. The group $G$ has order $168 =2^3 \cdot 3 \cdot 7$ and contains $21$ elements of order $2$, $56$ elements of order $3$, $42$ elements of order $4$ and $48$ elements of order $7$. Every maximal subgroup of $G$ is isomorphic to either $S_4$ or to an extension $C_7 \rtimes C_3$ (the normalizer of a Sylow $7$-subgroup). Let $\mathscr{P}$ be a partition of $G$. If $\mathscr{P}$ contains an $S_4$ then each other member of $\mathscr{P}$ has order at most $|G|/24=7$ so we need to cover the remaining $42-6=36$ elements of order $4$ using $18$ cyclic subgroups of order $4$; this exceeds the number of elements of order $2$ left to cover, which is $12$, and this is a contradiction because each cyclic subgroup of order $4$ contains precisely one element of order $2$. We deduce that no member of $\mathscr{P}$ is maximal of type $S_4$, in particular the only subgroups in $\mathscr{P}$ that can cover elements of order $4$ are of type $C_4$ or $D_8$. Observe that $\mathscr{P}$ cannot contain any subgroup of type $D_8$ because every element of $G$ of order $2$ is the square of some element of order $4$ and $D_8$ contains elements of order $2$ that are not the square of any of its elements of order $4$. In particular $\mathscr{P}$ contains all the $21$ cyclic subgroups of order $4$, which partition the elements of order $2$ and $4$. We now need to cover the elements of order $3$ and $7$ using subgroups of odd order (given that the elements of order $2$ are already partitioned). If $\mathscr{P}$ does not contain any $C_7 \rtimes C_3$ then it contains all the subgroups of order $3$ and $7$, giving a total of $|\mathscr{P}|=21+28+8=57$. If $\mathscr{P}$ contains a $C_7 \rtimes C_3$ then it contains only one such subgroup given that its index is smaller than its order. To finish covering the elements of order $3$ and $7$ we need $21+7=28$ subgroups (of order $3$ and $7$), proving that $|\mathscr{P}|=21+1+28=50$. This proves that $\rho(G) \geq 50$. On the other hand a partition of $G$ is given by one $C_7 \rtimes C_3$ (call it $H$), the $21$ $C_3$ not contained in $H$, the $7$ $C_7$ not contained in $H$ and the $21$ $C_4$. This gives the upper bound $\rho(G) \leq 50$ and proves the result.

\ 

\item $\rho(\PSL_2(9)) = 82=9^2+1$. Let $G=\PSL_2(9) \cong A_6$. The group $G$ has order $360=2^3 \cdot 3^2 \cdot 5$ and contains $45$ elements of order $2$, $80$ elements of order $3$, $90$ elements of order $4$ and $144$ elements of order $5$. Every maximal subgroup of $G$ is isomorphic to one of $S_4$, $A_5$ and $(C_3 \times C_3) \rtimes C_4$. Let $\mathscr{P}$ be a partition of $G$. If $\mathscr{P}$ contains an $A_5$ then each member of $\mathscr{P}$ has order at most $|G|/|A_5|=6$ so all the cyclic subgroups of order $4$ belong to $\mathscr{P}$; since $A_5$ contains elements of order $2$ this exceeds the number of elements of order $2$ in $G$. This shows that $\mathscr{P}$ does not contain any $A_5$. If $\mathscr{P}$ contains a $S_4$ then any other member of $\mathscr{P}$ has order at most $|G|/|S_4|=15$ so to cover the elements of order $4$ we need to use subgroups of type $C_4$ or $D_8$. Observe that $\mathscr{P}$ cannot contain any subgroup of type $D_8$ because every element of $G$ of order $2$ is the square of some element of order $4$ and $D_8$ contains elements of order $2$ that are not the square of any of its elements of order $4$. Since $S_4$ contains $6$ elements of order $4$, we need to cover the remaining $90-6=84$ elements of order $4$ using $84/2=42$ cyclic subgroups of order $4$, however each of them contains one element of order $2$ and together with the $9$ elements of order $2$ contained in $S_4$ this exceeds the total number of elements of order $2$ in $G$. This implies that no member of $\mathscr{P}$ is maximal of type $S_4$. If $\mathscr{P}$ does not contain any point stabilizer $(C_3 \times C_3) \rtimes C_4$ then any member of $\mathscr{P}$ contains at most $2$ elements of order $4$, this gives at least $90/2=45$ subgroups, furthermore since the $A_5$'s are not in $\mathscr{P}$ and the cyclic subgroups of order $4$ cover the elements of order $2$ we need all the $144/4=36$ cyclic subgroups of order $5$ to cover the elements of order $5$, this gives $45+36=81$ subgroups, which do not form a cover because they do not contain any element of order $3$. This implies that some $H \cong (C_3 \times C_3) \rtimes C_4$ belongs to $\mathscr{P}$ and any other member of $\mathscr{P}$ has order at most $|G|/36=10$, so $\mathscr{P}$ contains $(90-18)/2=36$ subgroups of type $C_4$ ($D_8$ is excluded by the above argument). We need additional $144/4=36$ subgroups to cover the elements of order $5$ and all of the $9$ Sylow $3$-subgroups of $G$ not contained in $H$, this gives $|\mathscr{P}| \geq 1+36+36+9=82$. On the other hand a partition of $G$ consists of one ${C_3}^2 \rtimes C_4$ (call it $H$), the $36$ $C_4$ not contained in $H$, the $36$ $C_5$ and the $9$ ${C_3}^2$ not contained in $H$, so $\rho(G) \leq 82$ proving the result.

\ 

\item $\rho(\PSL_2(11)) = 122=11^2+1$. Let $G=\PSL_2(11)$. The group $G$ has order $660 = 2^2 \cdot 3 \cdot 5 \cdot 11$ and it contains $55$ elements of order $2$, $110$ elements of order $3$, $264$ elements of order $5$, $110$ elements of order $6$ and $120$ elements of order $11$. Moreover every element of order $2$ is the third power of some element of order $6$, and every element of order $3$ is the square of some element of order $6$. The maximal subgroups of $G$ are isomorphic to one of $A_5$, $C_{11} \rtimes C_5$ and $D_{12}$. Let $\mathscr{P}$ be a partition of $G$. The only subgroups containing elements of order $6$ are of type $C_6$ or $D_{12}$, however every element of $G$ of order $2$ is the third power of some element of order $6$ and $D_{12}$ contains elements of order $2$ that are not the third power of any of its elements of order $6$, so the subgroups of type $D_{12}$ cannot belong to any partition. In particular $\mathscr{P}$ contains all the $55$ cyclic subgroups of order $6$, which partition the elements of order $2$, $3$ and $6$. We now need to cover the elements of order $5$ and $11$ using subgroups of order coprime to $6$ (given that the elements of order $2$ and $3$ are already partitioned). If $\mathscr{P}$ contains a subgroup $H$ isomorphic to $C_{11} \rtimes C_5$ then every other subgroup in $\mathscr{P}$ has order at most $|G|/55=12$ so $\mathscr{P}$ must contain all the subgroups of order $5$ and $11$ outside $H$, giving a total of $|\mathscr{P}|=55+1+55+11=122$. This also shows that $\rho(G) \leq 122$. If $\mathscr{P}$ does not contain any $C_{11} \rtimes C_5$ then it contains all the subgroups of order $5$ and $11$ giving a total of $|\mathscr{P}|=55+66+12=133$. This proves that $\rho(G) \geq 122$.

\end{itemize}

\section{The Suzuki groups} \label{sectsuz}

Let $G$ be the Suzuki group $Sz(q)$ where $q=2^{2m+1}$. We will prove that $\rho(G) > \sigma(G)$. We have (see \cite{lucido}) $$\sigma(G) = \frac{1}{2} q^2(q^2+1).$$The order of $G$ is $|G|=q^2(q-1)(q^2+1)$. Let $U$ be the subgroup of $G$ consisting of the unitriangular matrices, it is a Sylow $2$-subgroup of $G$ and $|U|=q^2$, $\exp(U)=4$. Let $H$ be the subgroup of $G$ consisting of the diagonal matrices. Then $H$ is isomorphic to $\mathbb{F}_q^{\ast}$ therefore $|H|=q-1$. Let $r=2^{m+1}$, so that $(q+r+1)(q-r+1)=q^2+1$. Let $T_1$ be a cyclic maximal torus of order $q+r+1$ and let $T_2$ be a cyclic maximal torus of order $q-r+1$. The set $\Psi$ consisting of all the conjugates of $U$, $H$, $T_1$, $T_2$ is a partition of $G$. The only maximal subgroup of $G$ containing $T_i$ is $N_i=N_G(T_i)=T_i \langle t_i \rangle$ for $i=1,2$, where $t_i$ is an element of order $4$ and $|N_i:T_i|=4$ for $i=1,2$. The subgroup $B=UH$ is a semidirect product between $U$ and $H$, and has the structure of Frobenius group with Frobenius kernel $U$ and Frobenius complement $H$. It is a maximal subgroup of $G$ of order $q^2(q-1)$ and $H$ is self-normalized in $B$, in particular $B$ contains $q^2$ conjugates of $H$. The intersection between any two conjugates of $B$ is a conjugate of $H$. Moreover the permutation action of $G$ with point stabilizer $B$ is $2$-transitive of degree $q^2+1$. M. S. Lucido \cite{lucido} proved that $$\mathscr{P} = \{{N_1}^g,{N_2}^g,B^x\ :\ g \in G,\ x \in G-B\}$$ is a minimal cover of $G$ consisting of maximal subgroups.

\ 

Let $\mathscr{P}$ be a partition of $G$. Observe that the only maximal subgroups containing conjugates ${T_i}^x$ of $T_i$ (for $i=1,2$) are the normalizers $N_G({T_i}^x)$, which contain ${T_i}^x$ as a subgroup of index $4$, so any two distinct subgroups of the form ${T_i}^x$ generate $G$. Since $T_i$ is cyclic, this proves that to cover these conjugates we need at least $$|G:N_G(T_1)|+|G:N_G(T_2)| = \frac{|G|}{4(q+r+1)}+\frac{|G|}{4(q-r+1)} = \frac{1}{2} q^2(q^2-1)$$ subgroups in the partition $\mathscr{P}$. Since $|H|$ is odd, and the conjugates of $T_i$ have trivial intersection with the conjugates of $H$, their normalizers also have trivial intersection with the conjugates of $H$. The only maximal subgroups containing conjugates $H^x$ are the corresponding Borel subgroups $B^x$, where $B=UH$, and the normalizers $N_G(H^x)$, in which $H^x$ has index $2$. In the partition $\mathscr{P}$ there is at most one conjugate of $B$ because any two conjugates of $B$ intersect in a conjugate of $H$. Let $K \in \mathscr{P}$, not a conjugate of $B$, which contains a conjugate $H^x$ of $H$. If $K$ contains more than one conjugate of $H$ then it is contained in a conjugate $B_0$ of $B$, which is a Frobenius group, hence $|B_0:K| > 2$, indeed if $|B_0:K|=2$ then $K$ would be normal in $B_0$, and this contradicts the fact that any normal subgroup of a Frobenius group either contains or is contained in the Frobenius kernel. Since $H^x$ is self-normalized in $B_0$, it is self-normalized in $K$ too hence $K$ contains at most $q^2/4$ conjugates of $H$. If $H^x \unlhd K$ then of course $K$ contains at most $q^2/4$ conjugates of $H$ as well. Call $\ell$ the number of subgroups in $\mathscr{P}$ containing conjugates of $H$. Since $H$ has $q^2(q^2+1)/2$ conjugates in $G$ we must have $$q^2+(\ell-1) \frac{q^2}{4} \geq \frac{1}{2} q^2(q^2+1)$$ hence $\ell \geq 2q^2-1$. We deduce that $$|\mathscr{P}| \geq \frac{1}{2} q^2(q^2-1)+2q^2-1 = \sigma(G)+q^2-1$$ therefore $\rho(G) > \sigma(G)$.


\begin{thebibliography}{30}

\bibitem{baer} R. Baer. Partitionen endlicher gruppen. Math. Z., 75:333--372, 1960.

\bibitem{barnes} D. W. Barnes, On complemented chief factors of finite soluble groups. Bull. Austral. Math. Soc. Vol 7 (1972), 101--104.

\bibitem{bryce} Bryce, R. A.; On a Theorem of Hughes and Thompson. Bull. Austral. Math. Soc. 50 (1994), 41--48.

\bibitem{bfs} Bryce, R. A.; Fedri, V.; Serena, L. Subgroup coverings of some linear groups. Bull. Austral. Math. Soc. 60 (1999), no. 2, 227--238.

\bibitem{cohn} J.H.E. Cohn, On $n$-sum groups; Math. Scand., 75(1) (1994), 44--58.

\bibitem{cubo} E. Detomi, A. Lucchini, On the Structure of Primitive $n$-Sum Groups; CUBO A Mathematical Journal Vol.10 n. 03 (195--210), 2008.

\bibitem{dickson} L. E. Dickson, Linear groups, with an exposition of the Galois field theory. Dover Publications (2003).

\bibitem{dh} K. Doerk, O. Hawkes, Finite Soluble Groups. De Gruyter (1992).

\bibitem{far} M. Farrokhi, D.G., Some results on the Partitions of Groups. Rend. Sem. Mat. Univ. Padova, Vol 125 (2011).

\bibitem{fs} T. Foguel, N. Sizemore, Partition Numbers of Finite Solvable Groups (preprint), to appear in  Advances in Group Theory and Applications.

\bibitem{irrcov} M. Garonzi, A. Lucchini, Irredundant and minimal covers of finite groups. Comm. Algebra 44 (2016), no. 4, 1722--1727.

\bibitem{gasch} W. Gasch\"utz, Die Eulersche Funktion endlicher aufl\"osbarer Gruppen, Illinois J. Math. 3 (1959), 469--476.

\bibitem{ht} D. R. Hughes, J. G. Thompson, The $H$-problem and the structure of $H$-groups. Pacific Journal of Mathematics (1959).

\bibitem{kar} G. Karpilovsky, Groups Representations Volume 1 Part B: Introduction to Group Representations
and Characters. Elsevier Science Publishers B.V. (1992).

\bibitem{lucido} M. Lucido, On the Covers of Finite Groups; Groups St. Andrews 2001 in Oxford. Vol. II, 395--399, London Math. Soc. Lecture. Note Ser., 305, Cambridge Univ. Press, Cambridge (2003).

\bibitem{marsym} Mar\'oti, Attila; Covering the symmetric groups with proper subgroups. J. Combin. Theory Ser. A 110 (2005), no. 1, 97--111.

\bibitem{passman} D. S. Passman; Permutation Groups, Benjamin, New York (1968).

\bibitem{sizemore} N. Sizemore, Group covers and partitions: covering and partition numbers. Master Thesis (2013).

\bibitem{suzuki} M. Suzuki. On a Finite group with a partition, Arch. Math., 12:241--274 (1961).

\bibitem{tom} Tomkinson, M. J. Groups as the union of proper subgroups. Math. Scand. 81 (1997), no. 2, 191--198.

\bibitem{zappa} G. Zappa, Partitions and other coverings of finite groups. Illinois Journal of Mathematics Volume 47, Number 1/2, Spring/Summer 2003, Pages 571--580.
\end{thebibliography}
\end{document}